\documentclass[11pt]{amsart}
\usepackage[T1]{fontenc}
\usepackage[latin5]{inputenc}
\usepackage{float}
\usepackage{amssymb}
\usepackage{amsfonts}
\usepackage[centertags]{amsmath}
\usepackage{amsthm}
\usepackage{newlfont}

\usepackage{booktabs}
\usepackage{multirow}

 \makeatletter
\theoremstyle{plain}
\newtheorem{theorem}{Theorem}[section]
\numberwithin{equation}{section}
\numberwithin{figure}{section}  
\theoremstyle{plain}

\theoremstyle{plain}
\newtheorem{corollary}[theorem]{Corollary} 
\theoremstyle{plain}
\newtheorem{definition}[theorem]{Definition}
\theoremstyle{plain}
\newtheorem{lemma}[theorem]{Lemma} 
\theoremstyle{plain}

\begin{document}
\title{On the Domain of Four-Dimensional Forward Difference Matrix in Some Double Sequence Spaces}
\author{Orhan Tu\v{g}}
\author{Eberhard Malkowsky}
\author{Viladimir Rako\v{c}evi\'{c}}
\author{Bipan Hazarika}
\address[Orhan Tu\v{g}]{Department of Mathematics Education, Tishk International University, Erbil, Iraq}
\email{orhan.tug@tiu.edu.iq}
\address[Eberhard Malkowsky]{Faculty of Management, University Union Nikola Tesla, 11000 Belgrade, Serbia}
\email{ema@pmf.ni.ac.rs; Eberhard.Malkowsky@math.uni-giessen.de}
\address[Viladimir Rako\v{c}evi\'{c}]{Department of Mathematics, Faculty of Sciences and Mathematics University of Ni\v{s}, Vi\v{s}egradska 33, 18000, Ni\v{s}-Serbia}
\email{vrakoc@bankerinter.net}
\address[Bipan Hazarika]{Department of Mathematics, Gauhati University, Gauhati, India.}
\email{bh\_gu@gauhati.ac.in}

\begin{abstract}
In this paper, we introduce some new double sequence spaces $\mathcal{M}_u(\Delta)$ and  $\mathcal{C}_{\vartheta}(\Delta)$, where $\vartheta\in\{bp,bp0,r,r0\}$ as the domains of the four-dimensional forward difference matrix in the double sequence spaces $\mathcal{M}_u$ and $\mathcal{C}_{\vartheta}$, respectively. Then we investigate some topological and algebraic properties. Moreover, we determine the $\alpha-$, $\beta(\vartheta)-$, and $\gamma-$duals of the new spaces $\mathcal{M}_u(\Delta)$ and  $\mathcal{C}_{\vartheta}(\Delta)$. Finally, we characterize four-dimensional matrix classes $(\lambda(\Delta),\mu)$ and $(\mu,\lambda(\Delta))$, where $\lambda=\{\mathcal{M}_u,\mathcal{C}_{\vartheta}\}$ and $\mu=\{\mathcal{M}_u,\mathcal{C}_{\vartheta}\}$.
\end{abstract}\maketitle

\section{Introduction}
By $\Omega:=\{x=(x_{mn}): x_{mn}\in\mathbb{C},~~\forall m,n\in\mathbb{N}\}$, we denote the set of all complex valued double sequences; $\Omega$ is a vector space with coordinatewise addition and scalar multiplication and any vector subspace of $\Omega$ is called a double sequence space. A double sequence $x=(x_{mn})$ is called convergent in Pringsheim's sense to a limit point $L$, if for every $\epsilon>0$ there exists a natural number $n_0=n_0(\epsilon)$ and $L\in\mathbb{C}$ such that $|x_{mn}-L|<\epsilon$ for all $m,n>n_0$,	where $\mathbb{C}$ denotes the complex field; this is denoted by $L=p-\lim_{m,n\to\infty}x_{mn}$. The space of all double sequences that are convergent in the Pringsheim sense is denoted by $\mathcal{C}_p$ which is a linear space with coordinatewise addition and scalar multiplication. M\`{o}ricz \cite{Moricz} proved that the double sequence space $\mathcal{C}_p$ is a complete seminormed space with the seminorm
\begin{eqnarray*}
	\|x\|_{\infty}=\lim_{N\to\infty}\sup_{m,n\geq N}|x_{mn}|.
\end{eqnarray*}

The space of all null double sequences in Pringsheim's sense is denoted by $\mathcal{C}_{p0}$. 

A double sequence $x=(x_{mn})$ of complex numbers is called bounded if $\|x\|_{\infty}=\sup_{m,n\in\mathbb{N}}|x_{mn}|<\infty$, where $\mathbb{N}=\{0,1,2,\cdots\}$, and the space of all bounded double sequences is denoted by $\mathcal{M}_u$, that is,
\begin{eqnarray*}
	\mathcal{M}_u:=\{x=(x_{mn})\in\Omega: \|x\|_{\infty}=\sup_{m,n\in\mathbb{N}}|x_{m,n}|<\infty\};
\end{eqnarray*}
it is a Banach space with the norm $\|\cdot\|_{\infty}$.

Unlike as in the case of single sequences there are double sequences which are convergent in Pringsheim's sense but unbounded. That is, the set $\mathcal{C}_p\setminus\mathcal{M}_u$ is not empty. Boos \cite{BJ} defined the sequence $x=(x_{mn})$ by
\begin{eqnarray*}
	x_{mn}=\left\{
	\begin{array}{ccl}
		n&, & m=0, n\in\mathbb{N} \\
		0&, & m\geq1, n\in\mathbb{N},
	\end{array}
	\right.
\end{eqnarray*}
which is obviously in $\mathcal{C}_p$, i.e., $p-\lim_{m,n\rightarrow\infty}x_{mn}=0$, but not in the set $\mathcal{M}_u$, i.e., $\|x\|_{\infty}=\sup_{m,n\in\mathbb{N}}|x_{mn}|=\infty$. Thus,  $x\in\mathcal{C}_p\setminus\mathcal{M}_u $.

We also consider the set $\mathcal{C}_{bp}$ of double sequences which are both convergent in Pringsheim's sense and bounded, that is,
\begin{eqnarray*}
	\mathcal{C}_{bp}:=\mathcal{C}_{p}\cap
	\mathcal{M}_{u}=
	\left\{x=(x_{mn})\in\mathcal{C}_{p}:
	\|x\|_{\infty}=
	\sup\limits_{m,n\in\mathbb{N}}
	|x_{mn}|<\infty\right\}.
\end{eqnarray*}
The set $\mathcal{C}_{bp}$ is a Banach space with the norm
\begin{eqnarray*}
	\|x\|_{\infty}=\sup_{m,n\in\mathbb{N}}|x_{mn}|<\infty.
\end{eqnarray*}

Hardy \cite{HGH} called a sequence in the space $\mathcal{C}_p$ regularly convergent if it is a convergent single sequence with respect to each index. We denote the set of such double sequences by $\mathcal{C}_r$, that is,
\begin{eqnarray*}
	\mathcal{C}_r:=\{x=(x_{mn})\in
	\mathcal{C}_{p}:
	\forall_{m\in\mathbb{N}}
	(x_{mn})_{m}\in c,\textit{ and }
	\forall_{n\in\mathbb{N}}
	(x_{mn})_{n}\in c\},
\end{eqnarray*}
where $c$ denotes the set of all convergent single sequences of complex
numbers. Regular convergence requires the boundedness of double sequences; this is the main difference between regular convergence and the convergence in Pringsheim's sense. We also use the notations  $\mathcal{C}_{bp0}=\mathcal{M}_{u}\cap\mathcal{C}_{p0}$ and $\mathcal{C}_{r0}=\mathcal{C}_{r}\cap\mathcal{C}_{p0}$.

Throughout the text, unless otherwise stated we mean by the summation $\sum_{kl}x_{kl}$ without limits run from $0$ to $\infty$ is $\sum_{k,l=0}^{\infty}x_{kl}$.

The space $\mathcal{L}_q$ of all absolutely $q-$summable double sequences was introduced by Ba\c{s}ar and Sever \cite{FBYS} as follows
\begin{eqnarray*}
	\mathcal{L}_q:=\left\{x=(x_{kl})\in\Omega:\sum_{k,l}|x_{kl}|^q<\infty\right\},\quad (1\leq q<\infty)
\end{eqnarray*}
which is a Banach space with the norm $\|\cdot\|_q$ defined by
\begin{equation*}
	\|x\|_q=\left(\sum_{k,l}|x_{kl}|^q\right)^{1/q}.
\end{equation*}
Moreover, Zeltser \cite{ZM} introduced the space $\mathcal{L}_u$ which is the special case of the space $\mathcal{L}_q$ for $q=1$.

The double sequence spaces $\mathcal{BS}$, $\mathcal{CS}_{\vartheta}$, where $\vartheta\in\{p, bp, r\}$, and $\mathcal{BV}$ were introduced by Altay and Ba\c{s}ar \cite{Altay}. The set $\mathcal{BS}$ of all double series whose sequences of partial sums are bounded is defined by
\begin{equation*}
	\mathcal{BS}=\left\{x=(x_{kl})\in\Omega:\sup_{m,n\in\mathbb{N}}|s_{mn}|<\infty\right\}
\end{equation*}
where the sequence $s_{mn}=\sum_{k,l=0}^{m,n}x_{kl}$ is the $(m,n)-th$ partial sum of the series. The series space $\mathcal{BS}$ is a Banach space with norm defined as
\begin{equation}\label{eq2.9999}
	\|x\|_{\mathcal{BS}}=\sup_{m,n\in\mathbb{N}}\left|\sum_{k,l=0}^{m,n}x_{kl}\right|,
\end{equation}
which is linearly isomorphic to the sequence space $\mathcal{M}_u$. The set $\mathcal{CS_{\vartheta}}$ of all series whose sequences of partial sums are $\vartheta-$convergent in Pringsheim's sense is defined by
\begin{equation*}
	\mathcal{CS_{\vartheta}}=\left\{x=(x_{kl})\in\Omega: (s_{mn})\in\mathcal{C_{\vartheta}}\right\}
\end{equation*}
where $\vartheta\in\{p, bp, r\}$. The space $\mathcal{CS}_p$ is a complete seminormed space with the seminorm defined by
\begin{equation*}
	\|x\|_\infty=\lim_{n\to\infty}\left(\sup_{k,l\geq n}\left|\sum_{i,j=0}^{k,l}x_{ij}\right|\right),
\end{equation*}
which is isomorphic to the  sequence space $\mathcal{C}_p$. Moreover, the sets $\mathcal{CS}_{bp}$ and $\mathcal{CS}_{r}$ are also Banach spaces with the norm (\ref{eq2.9999}) and the inclusion $\mathcal{CS}_{r}\subset \mathcal{CS}_{bp}$ holds. The set $\mathcal{BV}$ of all double sequences of bounded variation is defined by Altay and Ba\c{s}ar \cite{Altay} as follows
\begin{equation*}
	\mathcal{BV}=\left\{x=(x_{kl})\in\Omega: \sum_{k,l}\left|x_{kl}-x_{k-1,l}-x_{k,l-1}+x_{k-1,l-1}\right|<\infty\right\}.
\end{equation*}
The space $\mathcal{BV}$ is Banach space with the norm defined by
\begin{equation*}
	\|x\|_{\mathcal{BV}}=\sum_{k,l}\left|x_{kl}-x_{k-1,l}-x_{k,l-1}+x_{k-1,l-1}\right|,
\end{equation*}
which is linearly isomorphic to the space $\mathcal{L}_u$ of absolutely convergent double series. Moreover, the inclusions $\mathcal{BV}\subset \mathcal{C_{\vartheta}}$ and $\mathcal{BV}\subset \mathcal{M}_u$ strictly hold.

Let $E$ be any double sequence space. Then,
\begin{eqnarray*}
	&&dE:=\left\{x=(x_{kl})\in\Omega:\left\{\frac{1}{kl}x_{kl}\right\}_{k,l\in\mathbb{N}}\in E \right\},\\
	&&\int E:=\left\{x=(x_{kl})\in\Omega:\left\{klx_{kl}\right\}_{k,l\in\mathbb{N}}\in E \right\},\\
	&&E^{\beta(\vartheta)}:=\bigg\{a=(a_{kl})\in\Omega:\left\{a_{kl}x_{kl}\right\}\in \mathcal{CS}_{\vartheta},\textit{ for every }x=(x_{kl})\in E\bigg\},\\
	&&E^{\alpha}:=\bigg\{a=(a_{kl})\in\Omega:\left\{a_{kl}x_{kl}\right\}\in \mathcal{L}_{u},\textit{ for every }x=(x_{kl})\in E\bigg\},\\
	&&E^{\gamma}:=\bigg\{a=(a_{kl})\in\Omega:\left\{a_{kl}x_{kl}\right\}\in \mathcal{BS},\textit{ for every }x=(x_{kl})\in E\bigg\}.
\end{eqnarray*}
Therefore, let $E_1$ and $E_2$ are arbitrary double sequences with $E_2\subset E_1$ then the inclusions $E_1^{\alpha}\subset E_2^{\alpha}$, $E_1^{\gamma}\subset E_1^{\alpha}$ and $E_1^{\beta(\vartheta)}\subset E_1^{\alpha}$ hold. But the inclusion $E_1^{\gamma}\subset E_1^{\beta(\vartheta)}$ does not hold, since $\mathcal{C}_p\setminus\mathcal{M}_u$ is not empty.\\

Let 
$A=(a_{mnkl})_{m,n,k,l\in\mathbb{N} }$
be an infinite four--dimensional matrix and $E_{1}$, $E_{2}\in\Omega$. We write
\begin{equation}
	\label{eq1.a}
	y_{mn}=A_{mn}(x)=\vartheta-   
	\sum_{k,l}a_{mnk}x_{kl} 
	\mbox{ for each }m,n\in\mathbb{N}.
\end{equation}
We say that $A$ defines a matrix transformation from $E_{1}$ to $E_{2}$ if
\begin{eqnarray}\label{eq1.b}
	A(x)=(A_{mn}(x))_{m,n}\in E_{2}
	\mbox{ for all }x\in E_{1}.
\end{eqnarray}

The $\vartheta-$summability domain $E_A^{(\vartheta)}$ of a four-dimensional infinite matrix $A$ in a double sequence space $E$ is defined by
\begin{eqnarray*}\label{eq1.c}
	E_A^{(\vartheta)}=\left\{x=(x_{kl})\in\Omega: Ax=\left(\vartheta-\sum_{k,l}a_{mnkl}x_{kl}\right)_{m,n\in\mathbb{N}}\textit{exists and is in }E\right\},
\end{eqnarray*}
which is a sequence space. The above notation (\ref{eq1.b}) says that $A=(a_{mnkl})_{m,n,k,l\in\mathbb{N}}$ maps the space $E_1$ into the space $E_2$ if $E_1\subset (E_2)_A^{(\vartheta)}$ and we denote the set of all four-dimensional matrices that map the space $E_1$ into the space $E_2$ by $(E_1:E_2)$. Thus, $A\in (E_1:E_2)$ if and only if the double series on the right side of (\ref{eq1.b}) $\vartheta-$converges for each $m,n\in\mathbb{N}$, i.e, $A_{mn}\in (E_1)^{\beta(\vartheta)}$ for all $m,n\in\mathbb{N}$ and we have $Ax\in E_2$ for all $x\in E_1$. \\

Adams \cite{ACR} defined that the four-dimensional infinite matrix $A=(a_{mnkl})$ is a triangular matrix if $a_{mnkl}=0$ for $k>m$ or $l>n$ or both. We also say by \cite{ACR} that a triangular matrix $A=(a_{mnkl})$ is called a triangle if $a_{mnmn}\neq0$ for all $m,n\in\mathbb{N}$. One can be observed easily that if $A$ is triangle, then $E_A^{(\vartheta)}$ and $E$ are linearly isomorphic.\\

Wilansky \cite[Theorem 4.4.2, p. 66]{AW} defined that if $E$ is a sequence space, then the continuous dual $E_A^{*}$ of the space $E_A$ is given by
\begin{eqnarray*}
	E_A^{*}=\{f:f=g\circ A, g\in E^{*}\}.
\end{eqnarray*}

Zeltser \cite{ZM2} stated the notations of the double sequences $e^{kl}=(e_{mn}^{kl}), e^1, e_k$ and $e$ by
\begin{eqnarray*}
	e_{mn}^{kl}=\left\{
	\begin{array}{ccl}
		1&, & (k,l)=(m,n); \\
		0&, & otherwise.
	\end{array}
	\right.
\end{eqnarray*}

\begin{eqnarray*}
	&&e^1=\sum_ke^{kl};  \textit{the double sequence that all terms of l-th column are one and} \\
	&&\textit{other terms are zero},\\
	&&e_k=\sum_le^{kl};  \textit{the double sequence that all terms of k-th row are one and other}\\
	&&\textit{terms are zero},\\
	&&e=\sum_{kl}e^{kl}; \textit{the double sequence that all terms are one} \end{eqnarray*}
for all $k,l,m,n\in\mathbb{N}$.\\

The four-dimensional forward difference matrix $\Delta=(\delta_{mnkl})$ is defined by
\begin{eqnarray*}
	\delta_{mnkl}:=\left\{
	\begin{array}{ccl}
		(-1)^{m+n-k-l}&, & m\leq k\leq m+1,~ n\leq l\leq n+1, \\
		0&, & \textrm{otherwise}
	\end{array}\right.
\end{eqnarray*}
for all $m,n,k,l\in\mathbb{N}$. The $\Delta-$transform of a double sequence $x=(x_{mn})$ is given by
\begin{eqnarray}\label{eq1.3}
	y_{mn}:=\{\Delta x\}_{mn}=x_{mn}-x_{m+1,n}-x_{m,n+1}+x_{m+1,n+1}\notag
\end{eqnarray}
for all $m,n\in\mathbb{N}$. We shall briefly discuss $\Delta^{-1}$ which is the inverse of four-dimensional forward difference matrix $\Delta$, where $(\Delta^{-1}\Delta)(x_{kl})=x_{kl}$. Let $\Delta^{-1}y_{kl}=x_{kl}$. Then we can show that $x_{kl}$ is a finite summation of the original double sequence $y_{kl}$.
\begin{eqnarray}\label{eq1.33}
	\Delta(\Delta^{-1}y_{kl})=\Delta x_{kl}=x_{kl}-x_{k+1,l}-x_{k,l+1}+x_{k+1,l+1}.
\end{eqnarray}
If we write the equation (\ref{eq1.33}) for $y_{00}, y_{01},y_{10},...,y_{kl}$

\begin{eqnarray*}
	\Delta(\Delta^{-1}y_{00})&=&\Delta x_{00}=x_{00}-x_{10}-x_{01}+x_{11}\\
	\Delta(\Delta^{-1}y_{01})&=&\Delta x_{01}=x_{01}-x_{11}-x_{02}+x_{12}\\
	\Delta(\Delta^{-1}y_{10})&=&\Delta x_{10}=x_{10}-x_{20}-x_{11}+x_{21}\\
	\Delta(\Delta^{-1}y_{11})&=&\Delta x_{11}=x_{11}-x_{21}-x_{12}+x_{22}\\
	&\vdots&\\
	\Delta(\Delta^{-1}y_{kl})&=&\Delta x_{kl}=x_{kl}-x_{k+1,l}-x_{k,l+1}+x_{k+1,l+1}.\\
\end{eqnarray*}
Then we add the left hand sides up to $y_{00}+ y_{01}+y_{10}+...+y_{kl}$
\begin{eqnarray*}
	\sum_{i,j=0}^{k,l}y_{i,j}=x_{k+1,l+1}+x_{00}-x_{k+1,0}-x_{0,l+1}
\end{eqnarray*}
for all $k, l\in\mathbb{N}$. To be able to have $x_{kl}$ instead of having $x_{k+1,l+1}$ we must write it as
\begin{eqnarray}\label{eq2.22}
	x_{kl}=\sum_{i,j=0}^{k-1,l-1}y_{i,j}-x_{00}+x_{k,0}+x_{0,l}
\end{eqnarray}
for all $k, l\in\mathbb{N}$. With this result we can introduce the role of inverse four-dimensional forward difference operator $\Delta^{-1}$ on the double sequence $y_{kl}$, where $x_{kl}=\Delta^{-1}y_{kl}$, as the $(k-1,l-1)^{th}-$partial sum of the double sequence $y_{kl}$ plus arbitrary constants on the first row and the first column of the double sequence $x=(x_{kl})$.

\section{New double sequence spaces}
In this section, we introduce new double sequence spaces  $\mathcal{M}_u(\Delta)$,  $\mathcal{C}_{\vartheta}(\Delta)$, where $\vartheta\in\{bp,r\}$, as the matrix domains of the four-dimensional matrix of the forward differences in the sequence spaces $\mathcal{M}_u$ and $\mathcal{C}_{\vartheta}$ as follow;
\begin{eqnarray*}
	&&\mathcal{M}_u(\Delta):=\left\{x=(x_{kl})\in\Omega:\sup_{k,l\in\mathbb{N}}\left|y_{kl}\right|<\infty\right\},\\
	&&\mathcal{C}_{\vartheta}(\Delta):=\left\{x=(x_{kl})\in\Omega:\exists {L\in\mathbb{C}}\ni \vartheta-\lim_{k,l\to\infty}\left|y_{kl}-L\right|=0\right\},\\
	&&\mathcal{C}_{\vartheta 0}(\Delta):=\left\{x=(x_{kl})\in\Omega: \vartheta-\lim_{k,l\to\infty}\left|y_{kl}\right|=0\right\},
\end{eqnarray*}
where $y_{kl}=\Delta x_{kl}=(x_{kl}-x_{k+1,l}-x_{k,l+1}+x_{k+1,l+1})$ for all $k,l\in\mathbb{N}$.

\begin{theorem}\label{thm2.1}
	The spaces $\mathcal{M}_u(\Delta)$ and $\mathcal{C}_{\vartheta}(\Delta)$, where $\vartheta\in\{bp,bp0,r,r0\}$ are Banach spaces with the norm
	\begin{eqnarray}\label{eq2.1}
		\|x\|_{\mathcal{M}_u(\Delta)}&:=&|x_{k,0}+x_{0,l}-x_{00}|+\|\Delta x\|_{\mathcal{M}_u}\\
		\nonumber&:=&|x_{k,0}+x_{0,l}-x_{00}|+\sup_{k,l\in\mathbb{N}}\left|x_{kl}-x_{k+1,l}-x_{k,l+1}+x_{k+1,l+1}\right|.
	\end{eqnarray}
	
\end{theorem}

\begin{proof}
	The linearity of those spaces is clear. Suppose that $x^i=(x_{kl}^i)$ is a Cauchy sequence in the space $\mathcal{M}_u(\Delta)$ for all $k,l\in\mathbb{N}$. Then
	\begin{eqnarray*}
		\|x^i-x^j\|_{\mathcal{M}_u(\Delta)}&=&|(x_{k,0}^i-x_{k,0}^j)+(x_{0,l}^i- x_{0,l}^j)-(x_{00}^i-x_{00}^j)|\\
		&&+\sup_{k,l\in\mathbb{N}}|\Delta (x_{kl}^i-x_{kl}^j)|\to 0
	\end{eqnarray*}
	as $i,j\to\infty$. Thus, we obtain $|x_{kl}^i-x_{kl}^j|\to 0$ for $i,j\to\infty$ and for every $k,l\in\mathbb{N}$. Hence $x^i=(x_{kl}^i)$ is a Cauchy sequence in $\mathbb{C}$ for each $k,l\in\mathbb{N}$. Since $\mathbb{C}$ is complete, then it converges to a sequence $x=(x_{kl})$, i.e., we have
	\begin{eqnarray*}
		\lim_{i\to\infty}x_{kl}^i=x_{kl}
	\end{eqnarray*}
	for each $k,l\in\mathbb{N}$. Therefore, for every $\epsilon>0$, there exits a natural number $N(\epsilon)$, such that for all $i,j\geq N(\epsilon)$, and for all $k,l\in\mathbb{N}$ we have 
	
	\begin{eqnarray*}
		|x_{k,0}^i-x_{k,0}^j|<\frac{\epsilon}{4}, ~|x_{0,l}^i-x_{0,l}^j|<\frac{\epsilon}{4},~|x_{0,0}^i-x_{0,0}^j|<\frac{\epsilon}{4},~|\Delta (x_{kl}^i-x_{kl}^j)|<\frac{\epsilon}{4}.
	\end{eqnarray*}
	Moreover, 
	\begin{eqnarray*}
		&&\lim_{j\to\infty}|x_{k,0}^i-x_{k,0}^j|=|x_{k,0}^i-x_{k,0}|<\frac{\epsilon}{4},\\
		&&\lim_{j\to\infty}|x_{0,l}^i-x_{0,l}^j|=|x_{0,l}^i-x_{0,l}|<\frac{\epsilon}{4},\\
		&&\lim_{j\to\infty}|x_{0,0}^i-x_{0,0}^j|=|x_{0,0}^i-x_{0,0}|<\frac{\epsilon}{4},\\
		&&\lim_{j\to\infty}|\Delta (x_{kl}^i-x_{kl}^j)|=|\Delta (x_{kl}^i-x_{kl})|<\frac{\epsilon}{4}
	\end{eqnarray*}
	for all $i\geq N(\epsilon)$. Hence, we obtain that
	\begin{eqnarray*}
		\|x^i-x\|_{\mathcal{M}_u(\Delta)}&=&|(x_{k,0}^i-x_{k,0})+(x_{0,l}^i- x_{0,l})-(x_{00}^i-x_{00})|\\
		&&+\sup_{k,l\in\mathbb{N}}|\Delta (x_{kl}^i-x_{kl})|\\
		&\leq&|x_{k,0}^i-x_{k,0}|+|x_{0,l}^i- x_{0,l}|+|x_{00}^i-x_{00}|\\
		&&+\sup_{k,l\in\mathbb{N}}|\Delta(x_{kl}^i-x_{kl})|<\epsilon.
	\end{eqnarray*}
	Now we must show that $x\in\mathcal{M}_u(\Delta)$.
	\begin{eqnarray*}
		\sup_{k,l\in\mathbb{N}}|\Delta x_{kl}|&=&\sup_{k,l\in\mathbb{N}}|x_{kl}-x_{k+1,l}-x_{k,l+1}+x_{k+1,l+1}|\\
		&=&\sup_{k,l\in\mathbb{N}}\left|x_{kl}-x_{kl}^i+x_{kl}^i-x_{k+1,l}+x_{k+1,l}^i-x_{k+1,l}^i-x_{k,l+1}+x_{k,l+1}^i-x_{k,l+1}^i\right.\\
		&&\left.+x_{k+1,l+1}-x_{k+1,l+1}^i+x_{k+1,l+1}^i\right|\\
		&\leq&\sup_{k,l\in\mathbb{N}}\left|\Delta x_{kl}^i\right|+\sup_{k,l\in\mathbb{N}}|\Delta x_{kl}^i-\Delta x_{kl}|<\infty
	\end{eqnarray*}
	Hence $x=(x_{kl})\in\mathcal{M}_u(\Delta)$. This completes the proof.
	
\end{proof}

Let $\vartheta=\{bp,bp0,r,r0\}$. We define the operator $P$ form $\lambda(\Delta)$ into itself, where $\lambda\in\{\mathcal{M}_u,\mathcal{C}_{\vartheta}\}$ as 
\begin{eqnarray*}
	P:\lambda(\Delta)&\to&\lambda(\Delta)\\
	x&\to& Px=\begin{bmatrix}
		0 & 0 & 0 & 0 & \cdots\\
		0 & x_{11} & x_{12} & x_{13} & \cdots\\ 
		0 & x_{21} & x_{22} & x_{23} & \cdots\\
		0 & x_{31} & x_{32} & x_{33} & \cdots\\
		\vdots & \vdots & \vdots & \vdots & \ddots
	\end{bmatrix}
\end{eqnarray*}
for all $x=(x_{kl})\in\lambda(\Delta)$. Clearly $P$ is a linear and bounded operator on $\lambda(\Delta)$.

Now we show that the four-dimensional forward difference operator $\Delta$ is a linear homeomorphism.

\begin{eqnarray}\label{eq3.2}
	\Delta:P(\lambda(\Delta))&\to&\lambda\\
	\nonumber x&\to&\Delta x=y=(x_{kl}-x_{k+1,l}-x_{k,l+1}+x_{k+1,l+1})
\end{eqnarray}
where the set $P(\lambda(\Delta))$ is defined by
\begin{eqnarray*}
P(\lambda(\Delta)):=\{x=(x_{kl})\in\mathbb{C}:x\in\lambda(\Delta)~and~ x_{00}=x_{k,0}=x_{0,l}=0,\forall k,l\in\mathbb{N} \}\subset\lambda(\Delta)
\end{eqnarray*}
and
\begin{eqnarray*}
	\|x\|_{P(\lambda(\Delta))}=\|\Delta x\|_{\lambda}.
\end{eqnarray*}
Therefore, the spaces $P(\lambda(\Delta))$ and $\lambda$ are equivalent as topological spaces, and the $\Delta$ and $\Delta^{-1}$ are norm preserving and $\|\Delta\|=\|\Delta^{-1}\|=1$. We prove the following Lemma \ref{Lem:functional} for the case $\lambda=\mathcal{C}_{r0}$ by using the results in \cite[Theorem 5., Remark 3., P.132]{Moricz}. Since the proofs of the other cases are similar to that of following Lemma \ref{Lem:functional}, we left them as an exercise to the reader.

\begin{lemma}\label{Lem:functional}
	A linear functional $f_{\Delta}$ on $P(\mathcal{C}_{r0}(\Delta))$ is continuous if and only if there exists a double sequence $a=(a_{kl})_{k,l\geq 1}\in\mathcal{L}_u$ such that 
	\begin{equation}\label{func. eq.}
		f_{\Delta}(x)=\sum_{k,l=1}^{\infty}a_{kl}(\Delta x)_{kl}
	\end{equation}
	for all $x\in P(\mathcal{C}_{r0}(\Delta))$.
\end{lemma}

\begin{proof}
	First we show that $\Delta: P(\mathcal{C}_{r0}(\Delta))\to\mathcal{C}_{r0}$, $\Delta x_{kl}=x_{kl}-x_{k+1,l}-x_{k,l+1}+x_{k+1,l+1}$ with $x_{00}=x_{k,0}=x_{0,l}=0$ for each $k,l\in\mathbb{N}$ is an isometric linear isomorphism, that is, we prove that $\Delta$ is a bijection between $P(\mathcal{C}_{r0}(\Delta))$ and $\mathcal{C}_{r0}$ by $\Delta x_{kl}=x_{kl}-x_{k+1,l}-x_{k,l+1}+x_{k+1,l+1}$ with $x_{00}=x_{k,0}=x_{0,l}=0$ for each $k,l\in\mathbb{N}$. Linearity is clear. Moreover, $x=0$ whenever $\Delta x=0$, and hence $\Delta$ is injective. 
	Now suppose that $y=(y_{kl})\in\mathcal{C}_{r0}$, we define the sequence $x=(x_{kl})$ by $x_{kl}=\sum_{i,j=0}^{k-1,l-1}y_{ij}$ with $x_{00}=x_{k,0}=x_{0,l}=0$ for each $k,l\in\mathbb{N}$. Then we have,
	\begin{eqnarray*}
		\|x\|_{P(\mathcal{C}_{r0}(\Delta))}&=&\sup_{k,l\in\mathbb{N}}|\Delta x_{kl}|\\
		&=&\sup_{k,l\in\mathbb{N}}\left|\Delta\left(\sum_{i,j=0}^{k-1,l-1}y_{ij}\right) \right|\\
		&=&\sup_{k,l\in\mathbb{N}}\left|\sum_{i,j=0}^{k-1,l-1}y_{ij}-\sum_{i,j=0}^{k,l-1}y_{ij}-\sum_{i,j=0}^{k-1,l}y_{ij}+\sum_{i,j=0}^{k,l}y_{ij} \right|\\
		&=&\sup_{k,l\in\mathbb{N}}\left|\sum_{i,j=0}^{k-1,l-1}y_{ij}-\left(\sum_{j=0}^{l-1}y_{kj}+\sum_{i,j=0}^{k-1,l-1}y_{ij}\right)\right.\\
		&&\left.-\left(\sum_{i=0}^{k-1}y_{il}+\sum_{i,j=0}^{k-1,l-1}y_{ij}\right)\right.\\
		&&\left.+\left(\sum_{j=0}^{l-1}y_{kj}+\sum_{i=0}^{k-1}y_{il}+\sum_{i,j=0}^{k-1,l-1}y_{ij}+y_{kl}\right) \right|\\
		&=&\sup_{k,l\in\mathbb{N}}|y_{kl}|=\|y\|_{\infty}<\infty.
	\end{eqnarray*}
	It shows that $x\in P(\mathcal{C}_{r0}(\Delta))$ and consequently $\Delta$ is surjective and norm preserving. It completes the first part of the proof.
	
	Now suppose that $f_{\Delta}$ is a linear functional on $P(\mathcal{C}_{r0}(\Delta))$. If $f_{\Delta}$ is continuous, then $f_{\Delta}\circ\Delta^{-1}$ is a continuous linear functional on $\mathcal{C}_{r0}$. Then by \cite[Remark 3.]{Moricz} there exists a double sequence $a=(a_{kl})_{k,l\geq 1}\in\mathcal{L}_u$ such that
	\begin{equation*}
		f_{\Delta}\circ\Delta^{-1}(y)=\sum_{k,l=0}^{\infty}a_{kl}y_{kl}
	\end{equation*}
	for all $y\in\mathcal{C}_{r0}$. It gives 
	\begin{equation*}
		f_{\Delta}(x)=\left(f_{\Delta}\circ\Delta^{-1}\right)(\Delta x)=\sum_{k,l=0}^{\infty}a_{kl}(\Delta x)_{kl}
	\end{equation*}
	for all $x\in P(\mathcal{C}_{r0}(\Delta))$. Conversely, if $f_{\Delta}(x)=\sum_{k,l=1}^{\infty}a_{kl}(\Delta x)_{kl}$ for all $x\in P(\mathcal{C}_{r0}(\Delta))$ and for some $a=(a_{kl})\in\mathcal{L}_u$, then
	\begin{eqnarray*}
		\left|f_{\Delta}(x)\right|=\left|\sum_{k,l=0}^{\infty}a_{kl}(\Delta x)_{kl}\right|&\leq&\sum_{k,l=1}^{\infty}|a_{kl}||(\Delta x)_{kl}|\\
		&\leq&\|x\|_{ P(\mathcal{C}_{r0}(\Delta))}\sum_{k,l=0}^{\infty}|a_{kl}|\\
		&=&\|x\|_{P(\mathcal{C}_{r0}(\Delta))}\|a\|_{\mathcal{L}_u}.
	\end{eqnarray*}
	Therefore, $\|f_{\Delta}\|\leq\|a\|_{\mathcal{L}_u}$ and then we see that $f_{\Delta}$ is a bounded(continuous) linear functional on $P(\mathcal{C}_{r0}(\Delta))$. This completes the proof.
\end{proof}

\begin{definition}
	Let $X$ and $Y$ be Banach spaces, and $\mathcal{B}(X,Y)$ be the space of bounded linear operators from $X$ into $Y$. An operator $T\in \mathcal{B}(X,Y)$ is called an isometry if $\|Tx\|=\|x\|$ for all $x\in X$.
\end{definition}
Now we denote the continuous duals of $P(\lambda(\Delta))$ and $\lambda$ by $[P(\lambda(\Delta))]^*$ and $\lambda^*$, respectively. We may now show that the operator
\begin{eqnarray*}
	T:[P(\lambda(\Delta))]^*&\to& \lambda^*\\
	f_{\Delta}&\to&f=f_{\Delta}\circ(\Delta^{-1})
\end{eqnarray*}
is a linear isometry. Hence,
$[P(\mathcal{M}_u(\Delta))]^*\cong\mathcal{M}_u^*$, by \cite[Remark 3.]{Moricz} we have $[P(\lambda(\Delta))]^*\cong\lambda^*\cong\mathcal{L}_u$, where $\lambda\in\{\mathcal{C}_{r},\mathcal{C}_{r0}\}$, by \cite[Theorem 8.]{Moricz} we have $[P(\mu(\Delta))]^*\cong\mu^*\cong\ell_1(\ell_{\infty}^*)$, where $\mu\in\{\mathcal{C}_{bp},\mathcal{C}_{bp0}\}$, and the sets $\ell_1$ and $\ell_{\infty}$ represent absolutely summable and bounded single sequence spaces, respectively. 

Now we prove the following Theorem only for the case $\lambda=\mathcal{C}_{r0}$.

\begin{theorem}
	The continuous dual $[P(\mathcal{C}_{r0}(\Delta))]^*$ is isometrically isomorphic to $\mathcal{C}_{r0}^*\cong\mathcal{L}_u$.
\end{theorem}

\begin{proof}
	Let us define an operator 
	\begin{equation*}
		T:[P(\mathcal{C}_{r0}(\Delta))]^*\to\mathcal{C}_{r0}^*\cong\mathcal{L}_u
	\end{equation*}
	with $T(f_{\Delta})=\left(f_{\Delta}(e^{kl})\right)_{k,l\geq 1}$,
	\begin{equation*}
		T\left(f_{\Delta}(x)\right)=T\left(\left(f_{\Delta}\circ\Delta^{-1}\right)(\Delta x)\right)=\sum_{k,l=1}^{\infty}a_{kl}T((\Delta x)_{kl})
	\end{equation*}
	where $a=(a_{kl})\in\mathcal{L}_u$. Therefore, $T$ is a surjective linear map by Lemma \ref{Lem:functional}. Moreover, since $T(f_{\Delta}(e^{kl}))=0=(0,0,0,...)$ implies $f_{\Delta}=0$, where $(x_{kl})=e^{kl}$ is Schauder basis for $\mathcal{C}_{r0}$ by the definition of double Schauder basis \cite[Definition 4.2., p. 14]{Net convergence}, $T$ is injective. Let $f_{\Delta}\in[P(\mathcal{C}_{r0}(\Delta))]^*$  and $x\in P(\mathcal{C}_{r0}(\Delta))$. Then we have
	\begin{eqnarray*}
		\left|f_{\Delta}(x)\right|=\left|f_{\Delta}\left(\sum_{k,l=1}^{\infty}(\Delta x)_{kl}e^{kl}\right)\right|&=&\left|\sum_{k,l=1}^{\infty}(\Delta x)_{kl}f_{\Delta}(e^{kl})\right|\\
		&\leq&\sum_{k,l=1}^{\infty}\left|f_{\Delta}(e^{kl})\right||(\Delta x)_{kl}|\\
		&\leq&\sup_{k,l\in\mathbb{N}}|(\Delta x)_{kl}|\sum_{k,l=1}^{\infty}\left|f_{\Delta}(e^{kl})\right|\\
		&\leq&\|x\|_{P(\mathcal{C}_{r0}(\Delta))}\|T(f_{\Delta})\|_{\mathcal{L}_u}.
	\end{eqnarray*}
	Then we obtain 
	\begin{equation}\label{norm of func.(1)}
		\|f_{\Delta}\|_{\infty}\leq\|T(f_{\Delta})\|_{\mathcal{L}_u}.
	\end{equation}
	
	Furthermore, since $\left|f_{\Delta}(e^{kl})\right|\leq\|f_{\Delta}\|_{\infty}\|e^{kl}\|_{P(\mathcal{C}_{r0}(\Delta))}=\|f_{\Delta}\|_{\infty}$, then we have
	\begin{equation}\label{norm of func.(2)}
		\|T(f_{\Delta})\|_{\mathcal{L}_u}=\sup_{k,l\in\mathbb{N}}\left|f_{\Delta}(e^{kl})\right|\leq\|f_{\Delta}\|_{\infty}.
	\end{equation}
	We obtain by (\ref{norm of func.(1)}) and (\ref{norm of func.(2)}) that $\|T(f_{\Delta})\|_{\mathcal{L}_u}=\|f_{\Delta}\|_{\infty}$. This completes the proof.
\end{proof}

\section{Dual Spaces of the New Double Sequence Spaces}

In this section, we determine the $\alpha-$, $\beta(\vartheta)-$ and $\gamma-$duals of our new double sequence spaces. First, we begin with some lemmas to determine the $\alpha-$, $\beta(\vartheta)-$ and $\gamma-$duals of the spaces $\mathcal{M}_u(\Delta)$, $\mathcal{C}_{\vartheta}(\Delta)$, where $\vartheta\in\{bp,r\}$.

\begin{lemma}\label{lem3.10}
	We have $\sup_{k,l\in\mathbb{N}}|\Delta x_{kl}|<\infty$ if and only if 
	\begin{itemize}
		\item[(i)] $\sup_{k,l\in\mathbb{N}}\frac{1}{kl}|x_{kl}|<\infty$,
		\item[(ii)] $\sup_{k,l\in\mathbb{N}}\left|kl\Delta \left(\frac{1}{kl}x_{kl}\right)\right|<\infty$.
	\end{itemize}
\end{lemma}

\begin{proof}
	Suppose that there exists a positive real number $M$ such that
	\begin{equation*}
		\sup_{k,l\in\mathbb{N}}|x_{kl}-x_{k+1,l}-x_{k,l+1}+x_{k+1,l+1}|\leq M.
	\end{equation*}
	Then
	
	\begin{eqnarray*}
		|x_{kl}|=|x_{k,0}+x_{0,l}-x_{00}+x_{kl}|=\left|\sum_{i,j=0}^{k-1,l-1}\Delta x_{ij} \right|\leq \sum_{i,j=0}^{k-1,l-1}\left|\Delta x_{ij} \right|\leq M(kl). 
	\end{eqnarray*}
	It is clearly seen that (i) is necessary. Moreover, by considering the condition (i) there exists positive real numbers $N_1, N_2, N_3$ such that
	\begin{eqnarray}
		&&\sup_{k,l\in\mathbb{N}}\frac{1}{(k+1)l}|x_{k+1,l}|\leq N_1,\\
		&&\sup_{k,l\in\mathbb{N}}\frac{1}{k(l+1)}|x_{k,l+1}|\leq N_2,\\
		&&\sup_{k,l\in\mathbb{N}}\frac{1}{(k+1)(l+1)}|x_{k+1,l+1}|\leq N_3.
	\end{eqnarray}
	Then we have
	\begin{eqnarray*}
		kl\left|\Delta\left(\frac{1}{kl}x_{kl}\right)\right|&=&kl\left|\frac{1}{kl}x_{kl}-\frac{1}{(k+1)l}x_{k+1,l}-\frac{1}{k(l+1)}x_{k,l+1}\right.\\
		&&\left.+\frac{1}{(k+1)(l+1)}x_{k+1,l+1}\right|\\
		&=&kl\left|\frac{1}{kl}\Delta x_{kl}+\left(\frac{1}{kl(k+1)}x_{k+1,l}+\frac{1}{kl(l+1)}x_{k,l+1}\right.\right.\\
		&&\left.\left.-\frac{(k+l+1)}{kl(k+1)(l+1)}x_{k+1,l+1}\right)\right|\\
		&\leq&kl\left(\left|\frac{1}{kl}\Delta x_{kl}\right|+\left|\frac{1}{kl(k+1)}x_{k+1,l}\right|+\left|\frac{1}{kl(l+1)}x_{k,l+1}\right|\right.\\
		&&\left.+\left|\frac{(k+l+1)}{kl(k+1)(l+1)}x_{k+1,l+1}\right|\right)\\
		&\leq&M'
	\end{eqnarray*}
	where $M'=M+N_1+N_2+N_3$. So it gives the necessity of (ii).
	
	Now let us suppose that the conditions (i) and (ii) hold. By only considering the following inequality 
	\begin{eqnarray*}
		kl\left|\Delta\left(\frac{1}{kl}x_{kl}\right)\right|&=&\left|\frac{kl}{kl}x_{kl}-\frac{kl}{(k+1)l}x_{k+1,l}-\frac{kl}{k(l+1)}x_{k,l+1}\right.\\
		&&\left.+\frac{kl}{(k+1)(l+1)}x_{k+1,l+1}\right|\\
		&=&kl\left|\frac{1}{kl}\Delta x_{kl}-\left(\frac{1}{kl(k+1)}x_{k+1,l}+\frac{1}{kl(l+1)}x_{k,l+1}\right.\right.\\
		&&\left.\left.-\frac{(k+l+1)}{kl(k+1)(l+1)}x_{k+1,l+1}\right)\right|\\
		&\geq&|\Delta x_{kl}|-\left|-\frac{1}{(k+1)}x_{k+1,l}-\frac{1}{(l+1)}x_{k,l+1}\right.\\
		&&\left.+\frac{(k+l+1)}{(k+1)(l+1)}x_{k+1,l+1}\right|
	\end{eqnarray*}
	we can see the necessity of $\sup_{k,l\in\mathbb{N}}|\Delta x_{kl}|<\infty$.
	
\end{proof}

\begin{lemma}\label{lem3.1}
	Let $\Delta x_{kl}=y_{kl}$. If
	\begin{eqnarray*}
		\sup_{m,n\in\mathbb{N}}\left|\sum_{k,l=1}^{m,n}y_{kl}\right|<\infty
	\end{eqnarray*}
	then
	\begin{eqnarray*}
		\sup_{m,n\in\mathbb{N}}\left((m+1)(n+1)\left|\sum_{k,l=1}^{\infty} \frac{y_{m+k-1,n+l-1}}{(m+k)(n+l)}\right| \right)<\infty
	\end{eqnarray*}
\end{lemma}

\begin{proof}
	Let us consider Abel's double partial summation on the $(s,t)^{th}-$ partial sum of the series $\sum_{k,l=1}^{\infty} \frac{y_{m+k+1,n+l+1}}{(m+k)(n+l)}$ as in the following equation. 
	\begin{eqnarray}\label{eq3.1}
		\sum_{k,l=1}^{s,t} \frac{y_{m+k-1,n+l-1}}{(m+k)(n+l)}&=&\sum_{k,l=1}^{s,t} y_{m+k-1,n+l-1}\left(\frac{1}{(m+k)(n+l)}\right)\\
		\nonumber&=&\sum_{k,l=1}^{s-1,t-1} \left(\sum_{i,j=1}^{k,l}y_{m+i-1,n+j-1}\right)\Delta_{11}^{kl}\left(\frac{1}{(m+k)(n+l)}\right)\\
		\nonumber&&+\sum_{k=1}^{s-1} \left(\sum_{i,j=1}^{k,t}y_{m+i-1,n+j-1}\right)\Delta_{10}^{kl}\left(\frac{1}{(m+k)(n+t)}\right)\\
		\nonumber&&+\sum_{l=1}^{t-1} \left(\sum_{i,j=1}^{s,l}y_{m+i-1,n+j-1}\right)\Delta_{01}^{kl}\left(\frac{1}{(m+s)(n+l)}\right)\\
		\nonumber&&+\sum_{i,j=1}^{s,t}y_{m+i-1,n+j-1}\left(\frac{1}{(m+s)(n+t)}\right)
	\end{eqnarray}
	where for the double sequence $a_{kl}=\frac{1}{(m+k)(n+l)}$
	\begin{eqnarray*}
		&&\Delta_{10}^{kl}a_{kl}=a_{kl}-a_{k+1,l}\\
		&&\Delta_{01}^{kl}a_{kl}=a_{kl}-a_{k,l+1}\\
		&&\Delta_{11}^{kl}a_{kl}= \Delta_{10}^{kl}(\Delta_{01}^{kl}a_{kl})=\Delta_{01}^{kl}(\Delta_{10}^{kl}a_{kl})=a_{kl}-a_{k+1,l}-a_{k,l+1}+a_{k+1,l+1}.
	\end{eqnarray*}
	Since there exists a positive real number $M$ such that
	\begin{equation}
		\sup_{m,n\in\mathbb{N}}\left|\sum_{k,l=1}^{m,n}y_{kl}\right|\leq M,
	\end{equation}
	the equation (\ref{eq3.1}) is written as
	\begin{eqnarray*}	\sum_{k,l=1}^{s,t} \frac{y_{m+k-1,n+l-1}}{(m+k)(n+l)}&\leq&M\left[\sum_{k,l=1}^{s-1,t-1} \left(\frac{1}{(m+k)(n+l)}-\frac{1}{(m+k+1)(n+l)}\right.\right.\\
		&&\left.\left.-\frac{1}{(m+k)(n+l+1)}+\frac{1}{(m+k+1)(n+l+1)}\right)\right.\\&&+\left.\sum_{k=1}^{s-1}\frac{1}{(n+t)} \left(\frac{1}{(m+k)}-\frac{1}{(m+k+1)}\right)\right.\\
		&&+\left.\sum_{l=1}^{t-1}\frac{1}{(m+s)}\left(\frac{1}{(n+l)}-\frac{1}{(n+l+1)}\right)\right.\\&&+\left.\frac{1}{(m+s)(n+t)}\right]\\
		&=&\frac{M}{(m+1)(n+1)}.
	\end{eqnarray*}
	Therefore by passing to $\vartheta-$limit as $s,t\to\infty$, where $\vartheta=\{bp,r\}$, and taking supremum over $m,n\in\mathbb{N}$, then the condition
	\begin{eqnarray*}
		\sup_{m,n\in\mathbb{N}}\left((m+1)(n+1)\left|\sum_{k,l=1}^{\infty} \frac{y_{m+k-1,n+l-1}}{(m+k)(n+l)}\right| \right)<\infty
	\end{eqnarray*}
	is immediate.
\end{proof}

\begin{lemma}\label{lem3.2}
Let $\vartheta\in\{bp,r\}$. If the series $\sum_{k,l=1}^{\infty}\Delta x_{kl}$ is $\vartheta-$convergent, then
	\begin{eqnarray*}
		\vartheta-\lim_{m,n\to\infty}\left((m+1)(n+1)\left|\sum_{k,l=1}^{\infty} \frac{y_{m+k-1,n+l-1}}{(m+k)(n+l)}\right| \right)=0
	\end{eqnarray*}
\end{lemma}

\begin{proof}
	Since the partial sum of the series $\sum_{k,l=1}^{\infty}\Delta x_{kl}$ is $\vartheta-$convergent, where $\vartheta\in\{bp,r\}$, we have 
	\begin{eqnarray*}
		\left|\sum_{i,j=1}^{k,l}y_{m+i-1,n+j-1}\right|=\left|\sum_{i,j=m,n}^{m+k-1,n+l-1}y_{ij}\right|=O(1).
	\end{eqnarray*}
	Then by using the equality (\ref{eq3.1}) we write
	\begin{eqnarray*}
		(m+1)(n+1)\left|\sum_{k,l=1}^{\infty} \frac{y_{m+k-1,n+l-1}}{(m+k)(n+l)}\right| =O(1).
	\end{eqnarray*}
	If we let $\vartheta-$limit as $m,n\to\infty$, we reach the proof.
\end{proof}

\begin{corollary}\label{cor4.1}
	Let $\vartheta\in\{bp,r\}$ and $a=(a_{kl})$ be any double sequence. Then
	\begin{itemize}
		\item[(i)] If $\sup_{m,n\in\mathbb{N}}\left|\sum_{k,l=1}^{m,n}kla_{kl} \right|<\infty$, then
		\begin{eqnarray*}
			\sup_{m,n\in\mathbb{N}}\left|mn\sum_{k,l=m+1,n+1}^{\infty}a_{kl}\right|<\infty
		\end{eqnarray*}
		\item[(ii)] If $\sum_{k,l=1}^{\infty}kla_{kl}$ is $\vartheta-$convergent, then
		\begin{eqnarray*}
			\vartheta-\lim_{m,n\to\infty}\left(mn\sum_{k,l=m+1,n+1}^{\infty}a_{kl}\right)=0
		\end{eqnarray*}
		\item[(iii)] $\sum_{k,l=1}^{\infty}kla_{kl}$ is $\vartheta-$convergent if and only if
		\begin{eqnarray*}
			\sum_{k,l=1}^{\infty}R_{kl} \textit{ is $\vartheta-$convergent with }mnR_{mn}=O(1),
		\end{eqnarray*}
		where $R_{mn}=\sum_{k,l=m+1,n+1}^{\infty}a_{kl}$
	\end{itemize}
\end{corollary}

\begin{proof}
	The proof of (i) and (ii) can be easily seen by writing $kla_{kl}$ instead of $y_{kl}$ in Lemma \ref{lem3.1}, and writing $(k+1)(l+1)a_{k+1,l+1}$ instead of $y_{kl}$ in Lemma \ref{lem3.2}, respectively.
	
	To prove the corollary (iii), the following $(s,t)^{th}-$ partial sum can be written by using Abel's double summation formula that
	\begin{eqnarray*}
		\sum_{k,l=1}^{s,t}kla_{kl}&=&\sum_{k,l=1}^{s-1,t-1}\left(\sum_{i,j=0}^{k,l}a_{ij}\right)\Delta_{11}^{kl}(kl)+\sum_{k=1}^{s-1}\left(\sum_{i,j=0}^{k,t}a_{ij}\right)\Delta_{10}^{kl}(kl)\\
		&&+\sum_{i,j=0}^{s,t}a_{ij}(st)\\
		&=&\sum_{k,l=1}^{s,t}\left(\sum_{i,j=k,l}^{s,t}a_{ij}\right)+st\sum_{k,l=s+1,t+1}^{\infty}a_{kl}.
	\end{eqnarray*}
	Letting $\vartheta-$limit as $s,t\to\infty$, we obtain the statement in Part (iii).
\end{proof}

Let us define the following sets to be able to define the dual spaces of $\lambda(\Delta)$.

\begin{eqnarray*}
	&&D_1:=\int\mathcal{L}_u:=\left\{a=(a_{kl})\in\Omega: \sum_{k,l=1}^{\infty}kl|a_{kl}|<\infty \right\}\\
	&&D_2:=\int\mathcal{CS}_{\vartheta}:=\left\{a=(a_{kl})\in\Omega: \sum_{k,l=1}^{\infty}kla_{kl} ~~is~~\vartheta-convergent~\right\}\\
	&&D_3:=\int\mathcal{BS}:=\left\{a=(a_{kl})\in\Omega: \sum_{m,n}\left|\sum_{k,l=1}^{m,n}kla_{kl}\right|<\infty \right\}\\
	&&D_4:=\left\{a=(a_{kl})\in\Omega: \sum_{k,l=1}^{\infty}\left|\sum_{i,j=k,l}^{\infty}a_{ij}\right|<\infty \right\}
\end{eqnarray*}

\begin{theorem}
	Let $\lambda\in\{\mathcal{M}_u,\mathcal{C}_{bp},\mathcal{C}_r\}$. Then $\left[P(\lambda(\Delta))\right]^{\alpha}=D_1$
\end{theorem}

\begin{proof}
	We need to prove the existence of the inclusion relations $D_1\subset\left[P(\lambda(\Delta))\right]^{\alpha}$ and $\left[P(\lambda(\Delta))\right]^{\alpha}\subset D_1$. 
	
	Suppose that $a=(a_{kl})\in D_1$, i.e., $\sum_{k,l=1}^{\infty}kl|a_{kl}|<\infty$. Then by using Lemma \ref{lem3.10} we have
	\begin{eqnarray*}
		\sum_{k,l=1}^{\infty}|a_{kl}x_{kl}|=\sum_{k,l=1}^{\infty}kl|a_{kl}|\left(\frac{|x_{kl}|}{kl}\right)<\infty
	\end{eqnarray*}
	for all $x=(x_{kl})\in P(\lambda(\Delta))$. This shows that $a=(a_{kl})\in \left[P(\lambda(\Delta))\right]^{\alpha}$. Hence, the inclusion $D_1\subset\left[P(\lambda(\Delta))\right]^{\alpha}$ holds.
	
	Now suppose that  $a=(a_{kl})\in \left[P(\lambda(\Delta))\right]^{\alpha}$, i.e., $\sum_{k,l=1}^{\infty}|a_{kl}x_{kl}|<\infty$ for all $x=(x_{kl})\in P(\lambda(\Delta))$. If we consider the double sequence $x=(x_{kl})$ as
	\begin{eqnarray}\label{eq4.3}
		x_{kl}:=\left\{
		\begin{array}{cccl}
			0&, & k=0,l\geq0 \\
			0&, & l=0,k\geq0 \\
			kl&, & k\geq 1, l\geq 1
		\end{array}\right.
	\end{eqnarray}
	Then we have 
	\begin{eqnarray*}
		\sum_{k,l=1}^{\infty}|a_{kl}x_{kl}|=\sum_{k,l=1}^{\infty}kl|a_{kl}|<\infty
	\end{eqnarray*}
	which says $a=(a_{kl})\in D_1$. Hence, the inclusion $\left[P(\lambda(\Delta))\right]^{\alpha}\subset D_1$ holds. This concludes the proof.
\end{proof}

\begin{theorem}\label{thm3.7}
	Let $\lambda\in\{\mathcal{M}_u,\mathcal{C}_{bp}, \mathcal{C}_r\}$. Then $\left[P(\lambda(\Delta))\right]^{\beta(\vartheta)}=D_2\cap D_4$.
\end{theorem}

\begin{proof}
	We should show the validity of the inclusions $D_2\cap D_4\subset\left[P(\lambda(\Delta))\right]^{\beta(\vartheta)}$ and $\left[P(\lambda(\Delta))\right]^{\beta(\vartheta)}\subset D_2\cap D_4$.
	
	Suppose that the double sequence $a=(a_{kl})\in D_2\cap D_4$ and the sequence $x=(x_{kl})\in P(\lambda(\Delta))$ are defined with the relation (\ref{eq3.2}) between the terms of the sequence $x=(x_{kl})$ and $y=(y_{kl})$ as
	\begin{eqnarray}
		x_{kl}=\sum_{i,j=1}^{k,l}y_{i-1,j-1},
	\end{eqnarray}
	where $y=(y_{kl})\in\lambda$ which is defined as
	\begin{eqnarray}
		y_{kl}:=\left\{
		\begin{array}{ccccl}
			x_{11}&, & k=0,l=0 \\
			-x_{11}+x_{12}&, & k=0,l=1 \\
			-x_{11}+x_{21}&, & k=1, l=0\\
			x_{kl}-x_{k+1,l}-x_{k,l+1}+x_{k+1,l+1}&, &k\geq 1, l\geq 1
		\end{array}\right.
	\end{eqnarray}
	Then, we have the following $(s,t)^{th}-$partial sum of the series $\sum_{k,l}a_{kl}x_{kl}$ that
	\begin{eqnarray*}
		\sum_{k,l=1}^{s,t}a_{kl}x_{kl}&=&\sum_{k,l=1}^{s,t}a_{kl}\left(\sum_{i,j=1}^{k,l}y_{i-1,j-1} \right)\\
		&=&\sum_{k,l=1}^{s-1,t-1}\left(\sum_{i,j=k,l}^{s-1,t-1}a_{ij} \right)y_{kl}\\
		&=&\sum_{k,l=1}^{s-1,t-1}\left(\sum_{i,j=k,l}^{\infty}a_{ij} \right)y_{kl}-\sum_{k,l=1}^{s-1,t-1}\left(\sum_{i,j=s,t}^{\infty}a_{ij} \right)y_{kl}\\
		&=&\sum_{k,l=1}^{s-1,t-1}R_{kl}y_{kl}-R_{st}\sum_{k,l=1}^{s-1,t-1}y_{kl}.
	\end{eqnarray*}
	Now, by the Corollary \ref{cor4.1}(iii), we can say that the sequence $\sum_{k,l=1}^{s,t}a_{kl}x_{kl}$ is $\vartheta-$ convergent for every $x=(x_{kl})\in P(\lambda(\Delta))$, since $\sum_{k,l=1}^{s-1,t-1}R_{kl}y_{kl}$ is $\vartheta-$ convergent with $x_{st}R_{st}\to 0$ as $s,t\to\infty$. This yields that $a=(a_{kl})\in\left[P(\lambda(\Delta))\right]^{\beta(\vartheta)}$ and the inclusion $D_2\cap D_4\subset\left[P(\lambda(\Delta))\right]^{\beta(\vartheta)}$ holds.
	
	Now, suppose that $a=(a_{kl})\in\left[P(\lambda(\Delta))\right]^{\beta(\vartheta)}$. Then the series $\sum_{k,l=1}^{\infty}a_{kl}x_{kl}$ is $\vartheta-$convergent for every $x=(x_{kl})\in P(\lambda(\Delta))$. If we consider the sequence $x=(x_{kl})$ defined in (\ref{eq4.3})
	Then, we can observe that 
	\begin{eqnarray*}
		\sum_{k,l=1}^{\infty}a_{kl}x_{kl}=\sum_{k,l=1}^{\infty}kla_{kl}
	\end{eqnarray*}
	and by the equality $y=\Delta x$ we have the following series 
	\begin{eqnarray*}
		\sum_{k,l=1}^{s,t}kla_{kl}&=&\sum_{k,l=1}^{s-1,t-1}\left(\sum_{i,j=k,l}^{\infty}a_{ij} \right)-\sum_{k,l=1}^{s-1,t-1}\left(\sum_{i,j=s,t}^{\infty}a_{ij} \right)\\
		&=&\sum_{k,l=1}^{s-1,t-1}R_{kl}-stR_{st}
	\end{eqnarray*}
	which is $\vartheta-$convergent as $s,t\to\infty$. Thus, $a=(a_{kl})\in D_2$. Moreover, by Corollary \ref{cor4.1}(ii) we can write that $stR_{st}\to 0$ as $s,t\to\infty$ for every $y=(y_{kl})\in\lambda$, and  $\sum_{k,l=1}^{\infty}R_{kl}<\infty$. Therefore, $a=(a_{kl})\in D_4$. Hence the inclusion $\left[P(\lambda(\Delta))\right]^{\beta(\vartheta)}\subset D_2\cap D_4$ holds. This completes the proof.
\end{proof}

\begin{theorem}
	Let $\lambda\in\{\mathcal{M}_u,\mathcal{C}_{\vartheta}\}$. Then $\left[P(\lambda(\Delta))\right]^{\gamma}=D_3\cap D_4$, where $\vartheta\in\{bp,r\}$.
\end{theorem}

\begin{proof}
	The proof can be done with the similar path as above by considering Corollary \ref{cor4.1}(i). So, we omit the repetition.
\end{proof}

\section{Matrix Transformations}

In this section we characterize the four-dimensional matrix mapping from the sequence space $\lambda(\Delta)$ to $\mu$ and vice-versa. Then we conclude the section with some significant results.

\begin{theorem}\label{thm4.1}
The four-dimensional matrix $A=(a_{mnkl})\in(\lambda(\Delta):\mu)$ if and only if
\begin{eqnarray}
\label{eq4.1}A_{mn}=(a_{mnkl})_{k,l\in\mathbb{N}}\in\left(\lambda(\Delta)\right)^{\beta(\vartheta)}\textit{ for all } m,n\in\mathbb{N},\\
\label{eq4.2}A_{mn}(kl)=\sum_{k,l=1}^{\infty}kla_{mnkl}\in\mu,\\
\label{eq4.4}B=(b_{mnkl})\in(\lambda:\mu),
\end{eqnarray}
	where the four-dimensional matrix 
	\begin{equation}\label{matrix 4.1}
		B=(b_{mnkl})=\sum_{i,j=k,l}^{\infty}a_{mnij}\textit{ for all }m,n,k,l\in\mathbb{N}.
	\end{equation}
\end{theorem}

\begin{proof}
	Suppose that $A=(a_{mnkl})\in(\lambda(\Delta):\mu)$. Then, $A_{mn}(x)$ exists for every $x=(x_{kl})\in\lambda(\Delta)$ and is in $\mu$ for all $m,n\in\mathbb{N}$. If we define the sequence $x=(x_{kl})$ by
	\begin{eqnarray}
		x_{kl}:=\left\{
		\begin{array}{ccl}
			1&, & k=l \\
			0&, & \textit{otherwise}
		\end{array}\right.
	\end{eqnarray}
	for all $k,l\in\mathbb{N}$, then the necessity of (\ref{eq4.1}) is clear. If we define the sequence $x=(x_{kl})$ as $x_{kl}=kl$ for all $k,l\in\mathbb{N}$, then the necessity of (\ref{eq4.2}) is also clear by Theorem \ref{thm3.7}. Moreover, by Theorem \ref{thm3.7} we have $ \sum_{k,l=1}^{\infty}|a_{mnkl}|<\infty$ for each $m,n\in\mathbb{N}$. 
	
	Now suppose that $x=(x_{kl})\in P(\lambda(\Delta))\subset\lambda(\Delta)$ let us consider the $(s,t)^{th}-$partial sum of the series $\sum_{k,l=1}^{\infty}a_{mnkl}x_{kl}$ by considering the relation $x_{kl}=\sum_{i,j=0}^{k-1,l-1}y_{ij}$ between terms of the sequences $x=(x_{kl})$ and $y=(y_{kl})$ as in the following
	\begin{eqnarray*}
		A_{mn}^{st}(x)&=&\sum_{k,l=1}^{s,t}a_{mnkl}x_{kl}\\
		&=&\sum_{k,l=1}^{s,t}a_{mnkl}\left(\sum_{i,j=0}^{k-1,l-1}y_{ij}\right)\\
		&=&\sum_{k,l=1}^{s-1,t-1}\left(\sum_{i,j=k,l}^{s-1,t-1}a_{mnij}\right)y_{kl}\\
		&=&\sum_{k,l=1}^{s-1,t-1}\left(\sum_{i,j=k,l}^{\infty}a_{mnij}\right)y_{kl}-\sum_{k,l=1}^{s-1,t-1}\left(\sum_{i,j=s,t}^{\infty}a_{mnij}\right)y_{kl}\\
		&=&\sum_{k,l=1}^{s-1,t-1}b_{mnkl}y_{kl}-b_{mnst}\sum_{k,l=1}^{s-1,t-1}y_{kl}
	\end{eqnarray*}
	where $y\in\lambda$. We obtain by letting $\vartheta-$limit as $s,t\to\infty$ and by considering the Corollary \ref{cor4.1}$(iii)$ that $A_{mn}(x)=\sum_{k,l=1}^{\infty}b_{mnkl}y_{kl}$, that is $Ax=By$ for each $m,n\in\mathbb{N}$. Therefore,  $A=(a_{mnkl})\in(\lambda(\Delta):\mu)$ implies that $B=(b_{mnkl})\in(\lambda:\mu)$.
	
	Now suppose that the conditions (\ref{eq4.1})-(\ref{eq4.4}) hold. Let us take a sequence $x=(x_{kl})\in\lambda(\Delta)$ defined by
	\begin{eqnarray*}
		x_{kl}:=\left\{
		\begin{array}{cccl}
			x_{k,1}&, & k\geq1, l=1 \\
			x_{1,l}&, & k=l, l\geq1 \\
			\widetilde{x_{kl}}&, & k>l, l>1
		\end{array}\right.
	\end{eqnarray*}
	where $\widetilde{x}=(\widetilde{x_{kl}})\in P(\lambda(\Delta))$. Then, if we write again the above $(s,t)^{th}-$partial sum of the series $\sum_{k,l=1}^{\infty}a_{mnkl}x_{kl}$, we have 
	\begin{eqnarray*}
		A_{mn}^{st}(x)&=&\sum_{k,l=1}^{s,t}a_{mnkl}x_{kl}\\
		&=&a_{mn11}x_{11}+\sum_{l=2}^ta_{mn,1,l}x_{1,l}+\sum_{k=2}^sa_{mn,k,1}x_{k,1}+\sum_{k,l=2}^{s,t}a_{mnkl}\widetilde{x_{kl}}\\
		&=&a_{mn11}x_{11}+\sum_{k=2}^{s-1}b_{mn k,1}y_{k,1}+\sum_{l=2}^{t-1}b_{mn,1,l}y_{1,l}+\sum_{k,l=1}^{s-1,t-1}b_{mnkl}y_{kl}-b_{mnst}\sum_{k,l=1}^{s-1,t-1}y_{kl}.
	\end{eqnarray*}
	Therefore, we obtain by letting limit as $s,t\to\infty$ that 
	\begin{eqnarray*}
		A_{mn}(x)=a_{mn11}x_{11}+\sum_{k=2}^{\infty}b_{mn k,1}y_{k,1}+\sum_{l=2}^{\infty}b_{mn,1,l}y_{1,l}+\sum_{k,l=1}^{\infty}b_{mnkl}y_{kl}.
	\end{eqnarray*}
	Thus, $A_{mn}(x)$ exists for each $x=(x_{kl})\in\lambda(\Delta)$ and is in $\mu$ since $B\in(\lambda:\mu)$. This completes the proof.
\end{proof}

We list some four-dimensional matrix classes from and into the sequence spaces $\lambda,\mu=\{\mathcal{M}_u, \mathcal{C}_{bp},\mathcal{C}_r\}$ as in the following table, which have been characterized in some distinguished papers (see~\cite[Theorem 3.5]{OT2},\cite[Lemma 3.2]{CCBA},\cite[Theorem 2.2]{MZMMSAM},\cite[Theorem 3.2]{MYFBM}).
\begin{eqnarray}
	\label{eq3.0}&&\sup_{m,n\in\mathbb{N}}\sum_{k,l}|a_{mnkl}|<\infty,\\
	\label{eq3.01}&&\exists a_{kl}\in\mathbb{C} \ni \vartheta-\lim_{m,n\to\infty}a_{mnkl}=a_{kl}\textrm{ for all } k,l\in\mathbb{N},\\
	\label{eq3.99}&&\exists l\in\mathbb{C} \ni \vartheta-\lim_{m,n\to\infty}\sum_{k,l}a_{mnkl}=l\textrm{ exists },\\
	\label{eq3.9}&&\exists k_0\in\mathbb{N}\ni \vartheta-\lim_{m,n\to\infty}\sum_l|a_{mnk_0l}-a_{k_0l}|=0,\\
	\label{eq3.10}&&\exists l_0\in\mathbb{N}\ni \vartheta-\lim_{m,n\to\infty}\sum_k|a_{mnkl_0}-a_{kl_0}|=0,\\
	\label{eq3.7}&&\exists l_0\in\mathbb{N} \ni \vartheta-\lim_{m,n\to\infty}\sum_ka_{mnkl_0}=u_{l_0},\\
	\label{eq3.8}&&\exists k_0\in\mathbb{N} \ni \vartheta-\lim_{m,n\to\infty}\sum_la_{mnk_0l}=v_{k_0},\\
	\label{eq3.15}&&\exists a_{kl}\in\mathbb{C}\ni bp-\lim_{m,n\to\infty}\sum_{k,l}|a_{mnkl}-a_{kl}|=0,\\
	\label{eq3.151}&&bp-\lim_{m,n\to\infty}\sum_{l=0}^n a_{mnkl} \textrm{ exists for each } k\in\mathbb{N},\\
	\label{eq3.152}&&bp-\lim_{m,n\to\infty}\sum_{k=0}^m a_{mnkl} \textrm{ exists for each } l\in\mathbb{N},\\
	\label{eq3.153}&&\sum_{k,l}|a_{mnkl}| \textrm{ converges}.
\end{eqnarray}

\begin{table}[H]
	\centering
	\footnotesize
	\caption{The characterizations of the matrix classes $(\lambda;\mu)$, where $\lambda,\mu\in\{\mathcal{M}_u, \mathcal{C}_{bp},\mathcal{C}_r\}$.}
	\label{tab:t1}
	\begin{tabular}{ p{12em}p{5em}p{5em}p{5em} } 
		\toprule
		\boldmath{$From~\lambda{\downarrow}/To~\mu \to $}& \boldmath{$\mathcal{M}_u$} & \boldmath{$\mathcal{C}_{bp}$} & \boldmath{$\mathcal{C}_{r}$}  \\ 
		\midrule	
		$\mathcal{M}_u$ & 1 & 2 & * \\ 
		\midrule
		$\mathcal{C}_{bp}$ & 3 & 4 & 4\\ 
		\midrule
		$\mathcal{C}_{r}$ & * & 5 & 5\\  
		\bottomrule 
	\end{tabular}
\end{table}
\begin{flushleft}We list the necessary and sufficient conditions for each class in the following table. Note that $*$ shows the unknown characterization of respective four-dimensional matrix class.
\end{flushleft}

\begin{table}[H]
	\centering
	\footnotesize
	\caption{The necessary and sufficient conditions for $A\in(\lambda;\mu)$, where $\lambda,\mu\in\{\mathcal{M}_u, \mathcal{C}_{bp},\mathcal{C}_r\}$.}
	\label{tab:t2}
	\begin{tabular}{p{5em}p{5em}p{5em}p{5em}p{5em}} 
		\toprule
		\textbf{1 iff} & \textbf{2 iff} & \textbf{3 iff} & \textbf{4 iff} & \textbf{5 iff} \\ [0.5ex] 
		\midrule
		$(\ref{eq3.0})$& $(\ref{eq3.0})$ & $(\ref{eq3.0})$ & $(\ref{eq3.0})$ & $(\ref{eq3.0})$\\ & $(\ref{eq3.01})$ & & (\ref{eq3.01})&$(\ref{eq3.01})$ \\ 
		&$(\ref{eq3.15})$ & &(\ref{eq3.99}) & (\ref{eq3.99}) \\ 
		& $(\ref{eq3.151})$& & (\ref{eq3.9})& (\ref{eq3.7}) \\ 
		& $(\ref{eq3.152})$& & (\ref{eq3.10})& (\ref{eq3.8}) \\ 
		& $(\ref{eq3.153})$& & & \\ 
		\midrule
	\end{tabular}
\end{table}

\begin{corollary}
Let the four-dimensional matrix $B=(b_{mnkl})$ is defined as in (\ref{matrix 4.1}). Then the followings hold for four-dimensional infinite matrix $A=(a_{mnkl})$.
\begin{enumerate}
\item[(i)] $A\in(\mathcal{M}_u(\Delta),\mathcal{M}_u)$ if and only if the conditions in (\ref{eq4.1}) and (\ref{eq4.2}) hold, and $1$ holds in Table \ref{tab:t2} with $b_{mnkl}$ instead of $a_{mnkl}$.
\item[(ii)] $A\in(\mathcal{M}_u(\Delta),\mathcal{C}_{bp})$ if and only if the conditions in (\ref{eq4.1}) and (\ref{eq4.2}) hold, and $2$ holds in Table \ref{tab:t2} with $b_{mnkl}$ instead of $a_{mnkl}$.
\item[(iii)] $A\in(\mathcal{C}_{bp}(\Delta),\mathcal{M}_{u})$ if and only if the conditions in (\ref{eq4.1}) and (\ref{eq4.2}) hold, and $3$ holds in Table \ref{tab:t2} with $b_{mnkl}$ instead of $a_{mnkl}$.
\item[(iv)] Let $\vartheta=\{bp,r\}$. $A\in(\mathcal{C}_{bp}(\Delta),\mathcal{C}_{\vartheta})$ if and only if the conditions in (\ref{eq4.1}) and (\ref{eq4.2}) hold, and $4$ holds in Table \ref{tab:t2} with $b_{mnkl}$ instead of $a_{mnkl}$.
\item[(v)] Let $\vartheta=\{bp,r\}$. $A\in(\mathcal{C}_{r}(\Delta),\mathcal{C}_{\vartheta})$ if and only if the conditions in (\ref{eq4.1}) and (\ref{eq4.2}) hold, and $5$ holds in Table \ref{tab:t2} with $b_{mnkl}$ instead of $a_{mnkl}$.
\end{enumerate}
\end{corollary}

\begin{theorem}
The four-dimensional matrix $A=(a_{mnkl})\in(\mu:\lambda(\Delta))$ if and only if
	\begin{eqnarray}
\label{eq5.1}A_{mn}\in\mu^{\beta(\vartheta)},\\
\label{eq5.2}F=(f_{mnkl})\in(\mu:\lambda),
\end{eqnarray}
	where the four-dimensional matrix 
	\begin{equation}\label{matrix 4.2}
		F=(f_{mnkl})=\Delta_{11}^{mn}a_{mnij}=a_{mnij}-a_{m+1,nij}-a_{m,n+1,ij}+a_{m+1,n+1,ij}.
	\end{equation}
\end{theorem}

\begin{proof}
	Suppose that $A=(a_{mnkl})\in(\mu:\lambda(\Delta))$. Then, $A_{mn}(x)$ exists for every $x=(x_{kl})\in\mu$ and is in $\lambda(\Delta)$ for all $m,n\in\mathbb{N}$. Thus, the necessity of (\ref{eq5.1}) is immediate. Since $A_{mn}(x)\in\lambda(\Delta)$, then $\Delta A\in\lambda$ for every $x=(x_{kl})\in\mu$. Clearly $\Delta A$ is the matrix $F$. Hence, the necessity of the condition $F=(f_{mnkl})\in(\lambda:\mu)$ can be clearly seen. The  rest of the theorem can be followed by the similar path as in the Theorem \ref{thm4.1}. We omit the details.
\end{proof}

\begin{corollary}
Let the four-dimensional matrix $F=(f_{mnkl})$ is defined as in (\ref{matrix 4.2}). Then the followings hold for four-dimensional infinite matrix $A=(a_{mnkl})$.
\begin{enumerate}
	\item[(i)] $A\in(\mathcal{M}_u,\mathcal{M}_u(\Delta))$ if and only if the condition in (\ref{eq5.1}) holds, and $1$ holds in Table \ref{tab:t2} with $f_{mnkl}$ instead of $a_{mnkl}$.
	\item[(ii)] $A\in(\mathcal{M}_u,\mathcal{C}_{bp}(\Delta))$ if and only if the condition in (\ref{eq5.1}) holds, and $2$ holds in Table \ref{tab:t2} with $f_{mnkl}$ instead of $a_{mnkl}$.
	\item[(iii)] $A\in(\mathcal{C}_{bp},\mathcal{M}_{u}(\Delta))$ if and only if the condition in (\ref{eq5.1}) holds, and $3$ holds in Table \ref{tab:t2} with $f_{mnkl}$ instead of $a_{mnkl}$.
	\item[(iv)] Let $\vartheta=\{bp,r\}$. $A\in(\mathcal{C}_{bp},\mathcal{C}_{\vartheta}(\Delta))$ if and only if the condition in (\ref{eq5.1}) holds, and $4$ holds in Table \ref{tab:t2} with $f_{mnkl}$ instead of $a_{mnkl}$.
	\item[(v)] Let $\vartheta=\{bp,r\}$. $A\in(\mathcal{C}_{r},\mathcal{C}_{\vartheta}(\Delta))$ if and only if the condition in (\ref{eq5.1}) holds, and $5$ holds in Table \ref{tab:t2} with $f_{mnkl}$ instead of $a_{mnkl}$.
\end{enumerate}
\end{corollary}

\section{conclusion}
The four-dimensional backward difference matrix domain on some double sequence spaces has been studied by Demiriz and Duyar \cite{Demiriz}. Then Ba\c{s}ar and Tu\v{g} \cite{OT}, and Tu\v{g} \cite{OT2,Orhan,Orhan 2,OT3,Orhan 4,Orhan 6,Orhan new} studied the four-dimensional generalized backward difference matrix and its domain in some double sequence spaces. Moreover, Tu\v{g} at al. \cite{Orhan 7}, \cite{Orhan 8} studied the sequentially defined four-dimensional backward difference matrix domain on some double sequence spaces, and the space $\mathcal{BV}_{\vartheta 0}$ of double sequences of bounded variations, respectively.

In this work we defined the new double sequence spaces $\mathcal{M}_u(\Delta), \mathcal{C}_{\vartheta}(\Delta)$, where $\vartheta\in\{bp,r\}$ derived by the domain of four-dimensional forward difference matrix $\Delta$. Then we investigated some topological properties, determined $\alpha-$, $\beta(\vartheta)-$ and $\gamma-$duals and characterized some four-dimensional matrix classes related with these new double sequence spaces. 

The paper contribute nonstandard results and new contributions to the theory of double sequences. As a natural continuation of this work, the four-dimensional forward difference matrix domain in the double sequence spaces $\mathcal{C}_p$ and $\mathcal{L}_q$, where $0<q<\infty$ are still open problem. Moreover, the four-dimensional forward difference matrix domain in the spaces $\mathcal{C}_f$, $\mathcal{BS}$, $\mathcal{CS}$ and $\mathcal{BV}$ can be calculated. Furthermore, Hahn double sequence space can be defined and studied by using some significant results stated in this work.



%

\section*{Funding}
Not applicable.

\section*{Conflict of interest}
The authors declare that they have no conflict of interest.

\section*{Availability of data and material}
Not available.

\section*{Code availability}
Not available.

\end{document}